\documentclass[11pt]{amsart}
 \usepackage{amssymb,amsmath,amsfonts,epsfig,latexsym}

 \newtheorem{theorem}{Theorem}[section]
\newtheorem{definition}[theorem]{Definition}
\newtheorem{proposition}[theorem]{Proposition}
\newtheorem{lemma}[theorem]{Lemma}
\newtheorem{corollary}[theorem]{Corollary}

\theoremstyle{definition}
\newtheorem{remark}[theorem]{Remark}
\newtheorem{example}[theorem]{Example}

\def\span{\ensuremath{\mathrm{span}}}

\def\N{\ensuremath{\mathbb{N}}}
\def\Z{\ensuremath{\mathbb{Z}}}
\def\Q{\ensuremath{\mathbb{Q}}}
\def\P{\ensuremath{\mathbf{P}}}
\def\C{\ensuremath{\mathbb{C}}}
\def\A{\ensuremath{\mathbf{A}}}
\def\RR{\ensuremath{\mathbb{R}}}
\def\R{\ensuremath{\mathbf{R}}}

\def\S{\ensuremath{\mathbf{S}}}
\def\G{\ensuremath{\mathbb{G}}}

\def\O{\ensuremath{\mathcal{O}}}
\def\I{\ensuremath{\mathcal{I}}}

\def\m{\ensuremath{\mathfrak{m}}}

\def\an{\mathrm{an}}

\def\<{\ensuremath{\langle}}
\def\>{\ensuremath{\rangle}}

\DeclareMathOperator{\conv}{conv}

\DeclareMathOperator{\divisor}{div}
\DeclareMathOperator{\ev}{ev}

\DeclareMathOperator{\Hom}{Hom}
\DeclareMathOperator{\init}{in}

\DeclareMathOperator{\Span}{span}
\DeclareMathOperator{\Spec}{Spec}

\DeclareMathOperator{\trop}{Trop}
\DeclareMathOperator{\Trop}{\bf Trop}

\begin{document}

\title{Analytification is the limit of all tropicalizations}

\author[Payne]{Sam Payne}

\begin{abstract}
We introduce extended tropicalizations for closed subvarieties of toric varieties and show that the analytification of a quasprojective variety over a nonarchimedean field is naturally homeomorphic to the inverse limit of the tropicalizations of its quasiprojective embeddings.
\end{abstract}

\maketitle

\section{Introduction}

Let $K$ be an algebraically closed field that is complete with respect to a nontrivial nonarchimedean valuation $\nu:K^* \rightarrow \RR$.  The usual tropicalization associates to a closed subvariety $X$ in the torus $T = (K^*)^m$ the underlying set $\trop(X)$ of a finite polyhedral complex in $\RR^m$ of dimension $\dim X$, which is the closure of the image of $X(K)$ under the coordinatewise valuation map.  This construction is functorial; if $\varphi: T \rightarrow T'$ is a map of tori, then the corresponding linear map $\RR^m \rightarrow \RR^n$ takes $\trop(X)$ onto $\trop(X')$, where $X'$ is the closure of $\varphi(X)$.

Many varieties, including affine space $\A^m$ and all projective varieties, have no nonconstant invertible regular functions, and hence admit no nonconstant morphisms to tori.  However, all quasiprojective varieties have many closed embeddings in toric varieties, and toric varieties have a natural stratification by orbit closures of fixed dimension, whose locally closed pieces are disjoint unions of tori.  Here we associate to a closed subvariety $X$ in a toric variety an extended tropicalization $\Trop(X)$ with a natural stratification whose locally closed pieces are disjoint unions of the usual tropicalizations of intersections of $X$ with torus orbits.  This extended tropicalization construction is functorial with respect to torus-equivariant morphisms; the inverse limit of all extended tropicalizations of all embeddings of $X$ in toric varieties may be thought of, roughly speaking, as an intrinsic tropicalization of $X$.  Here we show that such inverse limits, for affine and quasiprojective varieties, are naturally homeomorphic to the nonarchimedean analytification of $X$, in the sense of Berkovich \cite{Berkovich93}.  The nonarchimedean analytification of an affine variety can be described in terms of multiplicative seminorms, as follows.

Recall that a multiplicative seminorm $| \ |$ on a ring $A$ is a map of multiplicative monoids from $A$ to $\RR_{\geq 0}$ that takes zero to zero and satisfies the triangle inequality $|f + g| \leq |f| + |g|.$  If $A$ is a $K$-algebra, then we say $|\ |$ is compatible with $\nu$ if
\[
|a| = \exp(-\nu(a)),
\]
for $a \in K$. 

Let $X$ be an affine algebraic variety over $K$.  The analytification $X^\an$ of $X$, in the sense of Berkovich, is the set of multiplicative seminorms on the coordinate ring $K[X]$ that are compatible with $\nu$ \cite[Remark~3.4.2]{Berkovich90}, equipped with the coarsest topology such that, for every $f \in K[X]$, the evaluation map sending a seminorm $| \ |$ to $|f|$ is continuous \cite[Section~2.4]{Baker07}.  The purpose of this note is to present a natural homeomorphism from $X^\an$ to the inverse limit of the tropicalizations of all affine embeddings of $X$, which we now describe in more detail, and to prove a similar result for quasprojective varieties (see Theorem~\ref{quasiprojective analytification}).

Extend $\nu$ to a map from $K$ to the extended real line
\[
\R = \RR \cup \infty
\]
taking the zero element of $K$ to $\infty$.  The extended real line has the topology in which the completed rays $(a,\infty]$, for $a \in \RR$, are a basis of neighborhoods for $\infty$, so the map taking $a$ to $\exp(-a)$ extends to a homeomorphism from $\R$ to $\RR_{\geq 0}$.  For positive integers $m$, let $\A^m = \Spec K[x_1, \ldots, x_m]$.  Associate to each point $y = (y_1, \ldots, y_m)$ in $\A^m(K)$ its tropicalization
\[
\Trop(y) = (\nu(y_1), \ldots, \nu(y_m))
\]
in $\R^m$, and write $\Trop: \A^m(K) \rightarrow \R^m$ for the extended map taking $y$ to $\Trop(y)$.  Now $\A^m$ carries a natural action of the torus $T^m = \Spec K [ x_1^{\pm 1}, \ldots, x_m^{\pm 1}]$, and if $\varphi: \A^m \rightarrow \A^n$ is an equivariant morphism with respect to some group morphism from $T^m$ to $T^n$ then the tropicalization of $\varphi(y)$ depends only on $\Trop(y)$, and the induced map from $\R^m$ to $\R^n$ is continuous.

For each affine embedding $\iota: X \hookrightarrow \A^m$, let the tropicalization of $X$ with respect to $\iota$ be
\[
\Trop(X, \iota) = \overline{ \{ \Trop(x) \ | \ x \in X(K) \} },
\]
the closure of the image of $X(K)$ in $\R^m$.  If $\jmath : X \hookrightarrow \A^n$ is another embedding and $\varphi: \A^m \rightarrow \A^n$ is an equivariant morphism such that $\jmath = \varphi \circ \iota$, then $\Trop(\varphi)$ maps $\Trop(X, \iota)$ into $\Trop(X, \jmath)$.  We write
\[
\varprojlim \Trop(X, \iota)
\]
for the inverse limit over all affine embeddings $\iota$ of $X$ and all such maps $\Trop(\varphi)$, in the category of topological spaces.

We give a natural homeomorphism from $X^\an$ to $\varprojlim \Trop(X, \iota)$ as follows.  We follow the usual notational convention, writing $x$ for a point in $X^\an$ and $| \ |_x$ for the corresponding seminorm on $K[X]$.  Let $\iota: X \hookrightarrow \A^m$ be an embedding given by $y \mapsto (f_1(y), \ldots, f_m(y))$ for some generators $f_1, \ldots, f_m$ of $K[X]$.  Then there is a natural continuous map $\pi_\iota$ from $X^\an$ to $\R^m$ given by
\[
\pi_\iota(x) = (-\log|f_1|_x ,\ldots, -\log|f_m|_x),
\]
where we define $-\log 0 = \infty$.  Furthermore, if $\jmath$ is an embedding of $X$ in $\A^n$ and $\varphi$ is an equivariant morphism from $\A^m$ to $\A^n$ such that $\jmath = \varphi \circ \iota$, then $\pi_\jmath = \Trop(\varphi) \circ \pi_\iota$.  Hence there is an induced map $\varprojlim \pi_\iota$ from $X^\an$ to the inverse limit of the spaces $\R^{m(\iota)}$ 
over all affine embeddings $\iota: X \hookrightarrow \A^{m(\iota)}$.

\begin{theorem} \label{main}
Let $X$ be an affine variety over $K$.  Then $\varprojlim \pi_\iota$ maps $X^\an$ homeomorphically onto $\varprojlim  \Trop(X, \iota)$.
\end{theorem}

\noindent In particular, $X^\an$ is an inverse limit of spaces with stratifications whose locally closed pieces are finite polyhedral complexes.  See Sections~\ref{affine tropicalizations} and \ref{toric tropicalization} for details on the structure of $\Trop(X, \iota)$.

Theorem~\ref{main} has been known in some form to nonarchimedean analytic geometers, but we have not been able to find this homeomorphism in the literature.  Some related ideas appeared in Berkovich's work on local contractibility of analytic spaces, in the language of pluri-stable formal schemes and their skeletons \cite{Berkovich99, Berkovich04}, and in Thuillier's work on analytifications of toric varieties and toroidal embeddings \cite{Thuillier07}.  Kontsevich and Soibelman have identified the analytification with an inverse limit of ``Clemens polytopes" of simple normal crossing models over the valuation ring, using theorems on existence of semistable reductions \cite{KontsevichSoibelman06}.  Similar inverse limit constructions with simple normal crossing resolutions appear in work of Boucksom, Favre, and Jonsson on valuations and singularities in several complex variables \cite{BoucksomFavreJonsson08}.  Arguments close to the spirit of this paper also appear in Gubler's elegant study of tropicalization of nonarchimedean analytic spaces \cite{Gubler07}.

We hope that the elementary algebraic presentation here will help open the ideas and results of analytic geometry to tropical geometers.  Relations to the basic tools of nonarchimdean analytic geometry, as developed for instance in \cite{BGR84}, should be useful for the development of rigorous algebraic foundations for tropical geometry, and may help explain recent results on the topology and geometry of tropicalizations of algebraic varieties, including theorems on singular cohomology \cite{Hacking08, HelmKatz08} and $j$-invariants of elliptic curves \cite{KatzMarkwigMarkwig08, Speyer07b, KatzMarkwigMarkwig08b}.

\noindent \textbf{Acknowledgments.}  I thank J.~Rabinoff and R.~Vakil for many stimulating conversations, and am grateful to M.~Baker, V.~Berkovich, B.~Conrad, M. Jonsson, D.~Stepanov, and the referee for helpful comments on earlier versions of this paper.

\section{Tropicalizations of affine embeddings} \label{affine tropicalizations}

We begin by recalling some of the basic definitions and properties of tropicalization.  See \cite[Section~2]{tropicalfibers} for details and further references.

The usual tropicalization map $\trop$ takes a point $y = (y_1, \ldots, y_m)$ in the torus $T^m(K)$ to its vector of valuations
\[
\trop(y) = (\nu(y_1), \ldots, \nu(y_m))
\]
in $\RR^m$.  If $X$ is a closed subvariety of $T^m$, then $\trop(X)$ is defined to be the closure of the image of $X(K)$ under $\trop$, which is the underlying set of an integral $G$-rational polyhedral complex of pure dimension equal to the dimension of $X$, where $G$ is the image of $K^*$ under $\nu$.  In other words, the polyhedral complex can be chosen so that each polyhedron is cut out by affine linear inequalities with linear coefficients in $\Z$ and constants in $G$.  Since $K$ is algebraically closed, $\trop(X(K))$ is exactly the set of $G$-rational points in $\trop(X)$.  Here we use an extended tropicalization map from $\A^m(K)$ to $\R^m$, following well-known ideas of Mikhalkin \cite{Mikhalkin06}, Speyer and Sturmfels \cite{SpeyerSturmfels04b}, and others.  In Section \ref{toric tropicalization}, we generalize this extended tropicalization map to arbitrary toric varieties.

As explained in the introduction, we extend the valuation $\nu$ to a map from $K$ to $\R$ by setting $\nu(0) = \infty$.  We extend $\trop$ similarly; if $y = (y_1, \ldots, y_m)$ is a point in $\A^m(K)$, then we define $\Trop(y)$ to be $(\nu(y_1), \ldots, \nu(y_m))$ in $\R^m$.  Now $\R^m$ is not a linear space in any usual sense.  However, $\R^m$ is the disjoint union of the linear spaces
\[
\RR^I = \{(v_1, \ldots, v_m) \in \R^m \ | \ v_i \mbox{ is finite if and only if } i \in I \},
\]
for $I \subset \{1, \ldots, m \}$, each of which is locally closed in $\R^m$.  Note that $\RR^I$ is in the closure of $\RR^J$ if and only if $I$ is a subset of $J$, and these linear spaces fit together to give a natural stratification
\[
\R^m_0 \subset \R^m_1 \subset \cdots \subset \R^m_m = \R^m,
\]
where $\R^m_i$ is the union of those $\RR^I$ such that the cardinality of $I$ is at most $i$.  Similarly, $\A^m$ is the disjoint union of the locally closed subvarieties
\[
T^I = \{(y_1, \ldots, y_m) \in \A^m \ | \ y_i \mbox{ is nonzero if and only if } i \in I \},
\]
and the extended tropicalization map is the disjoint union of the usual tropicalization maps from $T^I$ to $\RR^I$.  If $X$ is a closed subvariety of $\A^m$, then these stratifications induce a natural stratification of $\Trop(X)$, which is discussed in more detail and greater generality in Section~\ref{toric tropicalization}.

We begin the proof of Theorem~\ref{main} by showing that $\Trop(X, \iota)$ is the image of $\pi_\iota$.

\begin{lemma} \label{proper}
The natural projection $\pi_m$ from the analytification of $\A^m$ to $\R^m$ is proper.
\end{lemma}

\begin{proof}
The map $\pi_m$ is a product of $m$ copies of $\pi_1$, so it will suffice to show that $\pi_1$ is proper.  Now $\pi_1$ extends to a continuous map from the analytification of $\P^1$ to $\R \cup \{-\infty\}$, and this map is proper because the analytification of $\P^1$ is compact.  Therefore $\pi_1$ is the restriction of a proper map to the preimage of $\R$, and hence is proper.
\end{proof}

\begin{proposition} \label{image}
For any embedding $\iota: X \hookrightarrow \A^m$, the image of the induced map
\[
\pi_\iota: X^\an \longrightarrow \R^m
\]
is exactly $\Trop(X,\iota)$.
\end{proposition}

\begin{proof}
The projection $\pi_\iota$ is proper since it is the restriction of the proper map $\pi_m$ to the closed subset $X^\an$, and $X(K)$ is dense in $X^\an$.  Therefore, the image $\pi_\iota (X^\an)$ is exactly the closure of the set of tropicalizations of $K$-points in $X$, which is $\Trop(X, \iota)$.
\end{proof}

\begin{proof}[Proof of Theorem~\ref{main}]
The topology on $\varprojlim \Trop(X, \iota)$ is the coarsest such that the restrictions of the coordinate projections on $\R^{m(\iota)}$ are continuous, for all embeddings $\iota$ of $X$.  It follows that the topology on $X^\an$ is the coarsest such that $\varprojlim \pi_\iota$ is continuous.  Therefore, to prove that $\varprojlim \pi_\iota$ is a homeomorphism onto $\varprojlim \Trop(X, \iota)$, it will suffice to show that it is bijective.

We first show that $\varprojlim \pi_\iota$ is injective.  Suppose $\varprojlim \pi_\iota (x) = \varprojlim \pi_\iota(x')$.  We must show that $|f|_x = |f|_{x'}$ for any $f$ in the coordinate ring $K[X]$.  Choose any generating set $f_1, \ldots, f_m$ for $K[X]$, with $f_1 = f$, and let $\iota$ be the induced embedding of $X$.  Then $\pi_\iota(x) = \pi_\iota(x')$ by hypothesis, and projection to the first coordinate shows that $|f|_x = |f|_{x'}$, as required.

It remains to show that $\varprojlim \pi_\iota$ is surjective.  Let $(y_\iota)$ be a point in $\varprojlim \Trop(X, \iota)$.  Define a point $x \in X^\an$, as follows.  For each $f \in K[X]$, choose a generating set $f_1, \ldots, f_m$ for $K[X]$ with $f_1 = f$, and let $\iota: X \hookrightarrow \A^m$ be the corresponding embedding.  Then define $|f|_x$ to be the exponential of the negative of the first coordinate of $y_\iota \in \R^m$.  If $\jmath$ is another such embedding, given by a generating set $g_1, \ldots, g_n$ with $g_1 = f$ then $\iota \times \jmath$ embeds $X$ in $\A^{m + n}$, and there is an automorphism of $\iota \times \jmath$ that transposes the first and $(m+1)$th coordinates.  Since $(y_\iota)$ is an inverse system, projecting to $\A^m$ and $\A^n$ shows that the first coordinates of $y_\iota$ and $y_\jmath$ are equal, so $| \ |_x$ is well-defined.  It is straightforward to check the multiplicative property $|f \cdot g|_x = |f|_x \cdot |g|_x$ and the triangle inequality $| f + g|_x \leq |f |_x + |g|_x$ by considering any affine embedding of $X$ defined by a generating set for $K[X]$ that contains both $f$ and $g$.  So $| \ |_x$ is a multiplicative seminorm that is compatible with $\nu$, and the image of $x$ under $\varprojlim \pi_\iota$ is $(y_\iota)$, by construction.  Therefore, $\varprojlim \pi_\iota$ is surjective, and the theorem follows.
\end{proof}

\section{Tropicalizations of toric embeddings} \label{toric tropicalization}

In this section, we generalize the extended tropicalizations for embeddings of a variety in affine space to embeddings in arbitrary toric varieties.  This construction is applied in Section~\ref{quasiproj section} to study the analytifications of quasiprojective varieties.

Let $N \cong \Z^n$ be a lattice, and let $\Delta$ be a fan in $N_\R = N \otimes_\Z \R$, with $Y = Y(\Delta)$ the corresponding toric variety.  Let $M = \Hom(N,\Z)$ be the lattice dual to $N$, which is the lattice of characters of the dense torus $T \subset Y$.  See \cite{Fulton93} for standard notation and background for toric varieties.  We construct a space $\Trop(Y)$ with a functorial tropicalization map
\[
Y(K) \rightarrow \Trop(Y)
\]
as follows.  For each cone $\sigma \in \Delta$, let $N(\sigma) = N_\R / \Span(\sigma)$.  As a set, $\Trop(Y)$ is a disjoint union of linear spaces
\[
\Trop(Y) = \coprod_{\sigma \in \Delta} N(\sigma).
\]
Now $Y$ is a disjoint union of tori $T_\sigma$, where $T_\sigma$ is the torus whose lattice of one parameter subgroups is the image of $N$ in $N(\sigma)$.  In other words, $T_\sigma$ is the unique quotient of the dense torus $T$ that acts simply transitively on the orbit corresponding to $\sigma$.  Then the tropicalization map from $Y(K)$ to $\Trop(Y)$ is the disjoint union of the ordinary tropicalization maps from $T_\sigma(K)$ to $N(\sigma)$.

We now describe the topology on $\Trop(Y)$, considering first the affine case.  Let $U_\sigma$ be an affine toric variety.  Recall that the coordinate ring of $U_\sigma$ is the semigroup ring $K[S_\sigma]$, where $S_\sigma = \sigma^\vee \cap M$ is the multiplicative monoid of characters of $T$ that extend to regular functions on $U_\sigma$.  The preimage of $\RR$ under a monoid homomorphism from $S_\sigma$ to $\R$ is $\tau^\perp \cap S_\sigma$, for some face $\tau \preceq \sigma$.  Therefore, the disjoint union $\coprod_{\tau \preceq \sigma} N(\tau)$ is naturally identified with $\Hom(S_\sigma, \R)$,
where $v \in N(\tau)$ corresponds to the monoid homomorphism $\phi_v: S_\sigma \rightarrow \R$ given by
\[
\phi_v(u) = \left \{ \begin{array}{ll} \<u,v\> & \mbox{ if } u \in \tau^\perp \\
\infty & \mbox{ otherwise.}
\end{array} \right.
\]
This gives a natural identification
\[
\Trop(U_\sigma) = \Hom(S_\sigma, \R),
\]
and we give $\Trop(U_\sigma)$ the induced topology, as a subspace of $\R^{S_\sigma}$.

\begin{remark}
The monoid $S_\sigma$ is finitely generated.  Any choice of generators gives an embedding of $\Hom(S_\sigma, \R)$ in $\R^m$, and $\Trop(U_\sigma)$ carries the subspace topology.  Equivalently, a choice of generators for $S_\sigma$ gives a closed embedding of $U_\sigma$ in $\A^m$, and $\Trop(U_\sigma)$ is the tropicalization of this embedding.
\end{remark}

Suppose the toric variety is affine, and hence isomorphic to some $U_\sigma$.  The tropicalization map from $U_\sigma(K)$ to $\Trop(U_\sigma)$ can be interpreted in terms of these $\Hom$ spaces as follows.  The $K$-points of $U_\sigma$ correspond naturally and bijectively, through evaluation maps, to homomorphisms from $S_\sigma$ to the multiplicative monoid of $K$.  If $y$ is point in $U_\sigma(K)$, composing the evaluation map $\ev_y$ with the extended valuation $\nu: K \rightarrow \R$ gives a monoid homomorphism $\Trop(y)$ from $S_\sigma$ to $\R$.  The preimage of $\RR$ under this map is the intersection of $S_\sigma$ with $\tau^\perp$, where $\tau$ is the face of $\sigma$ corresponding to the orbit that contains $y$, and the induced map $(\tau^\perp \cap M) \rightarrow \RR$ is the ordinary tropicalization map for the closed embedding of $X \cap T_\tau$ in $T_\tau$.  Furthermore, there is a natural continuous and proper map from $U_\sigma^\an$ to $\Trop(U_\sigma)$ that takes a point $y$ to the monoid homomorphism $u \mapsto -\log | \chi^u|_y$.

We now consider the case where the toric variety $Y(\Delta)$ is not necessarily affine.  If $\sigma \in \Delta$ is a cone and $\tau$ is a face of $\sigma$, then $\Hom(S_\tau, \R)$ is canonically identified with the topological submonoid of $\Hom(S_\sigma, \R)$ consisting of those maps for which the image of $\tau^\perp \cap M$ is contained in $\RR$, and we define $\Trop(Y)$ to be the topological space defined by gluing along these identifications.  The natural maps from $U_\sigma^\an$ to $\Trop(U_\sigma)$ glue together to give a continuous and proper map from $Y^\an$ to $\Trop(Y)$.

Let $m$ be the dimension of $Y$.  There is also a natural stratification
\[
\Trop(Y)_0 \subset \cdots \subset \Trop(Y)_m = \Trop(Y),
\]
where $\Trop(Y)_i$ is the union of the vector spaces $N(\sigma)$ of dimension at most $i$, which are exactly those $N(\sigma)$ such that $\dim(\sigma) \geq m-i$.

\begin{example}
Suppose $X = \A^m$ is affine space.  Then $\Trop(X) = \Hom (\N^m, \R)$, by definition, which is naturally identified with $\R^m$.  In the stratification above, $\Trop(X)_i$ is the union of the coordinate subspaces $\R^I = \Hom(\N^I, \R)$ for subsets $I \subset \{1, \ldots, m \}$ of cardinality at most $i$.  In particular, this definition of tropicalization of toric varieties agrees with the definition of tropicalization of affine space in Section~\ref{affine tropicalizations}.
\end{example}

\begin{remark}
Roughly speaking, for any $\tau \preceq \sigma$, the tropicalization map may be thought of as a generalized moment map that is independent of polarization, with a stratification that corresponds to the stratification of a polytope $P$ by the unions of its faces of fixed dimension.  Suppose $Y$ is projective and $L$ is an ample $T$-equivariant line bundle on $Y$.  Then for each maximal cone $\sigma \in \Delta$ there is a unique character $u_\sigma \in M$ such that $L|_{U_\sigma}$ is equivariantly isomorphic to $\O(\divisor \chi^{u_\sigma})$.  The algebraic moment map $\mu$ from $Y(K)$ to $M_\RR$ sends a point $y$ in the dense torus $T$ to
\[
\mu(y) = \frac{\sum_\sigma | \chi^{u_\sigma}(y)| \cdot u_\sigma}{\sum_\sigma |\chi^{u_\sigma}(y)|}.
\]
Then $\mu(y)$ depends only on $\Trop(y)$, and $\mu$ extends to a homeomorphism from $\Trop(Y)$ onto the moment polytope $P = \conv \{u_\sigma \}$, and $\Trop(Y)_i$ is exactly the preimage of the union of the $i$-dimensional faces of $P$.  Compactifications of amoebas in moment polytopes were introduced in \cite{GKZ}, and have been studied in many subsequent papers.  The extended tropicalization map is more convenient than moment maps in some contexts due to its covariant functorial properties, independence of polarization, and the integral structures on the vector spaces $N(\sigma)$.
\end{remark}

\begin{remark} 
It is sometimes helpful to think of the topology on $\Trop(Y)$ locally near a point $v$ in $N(\sigma)$.  Roughly speaking, a sequence of points in $N(\tau)$ converges to $v$ if their projected images converge to $v$ in $N(\sigma)$ and they move toward infinity in the cone of directions specified by $\sigma$.  More precisely, the topology on $\Trop(Y)$ is determined by the following basis.  Choose a finite set of generators $u_1, \ldots, u_r$ for the monoid $S_\sigma$, and note that $u_i$ can be evaluated on $N(\tau)$ provided that $u_i$ is in $\tau^\perp$.  For each open set $U \subset N(\sigma)$ and positive number $n$, let $C(U, n)$ be the truncated cylinder
\[
C(U,n) = \bigcup_{\tau \preceq \sigma} \big \{ w \in N(\tau) \ | \ \pi(w) \in U \mbox{ and } \< u_i, w \> > n \mbox{ for } u_i \in \tau^\perp \smallsetminus \sigma^\perp \big \},
\]
where $\pi: N(\tau) \rightarrow N(\sigma)$ is the canonical projection.  We claim that these truncated cylinders are a basis for the topology on $\Trop(Y)$.  To see this, note that $\Trop(U_\sigma)$ is an open neighborhood of any point $v$ in $N(\sigma)$.  The choice of generators for $S_\sigma$ determines an embedding of $\Trop(U_\sigma)$ in $\R^r$, and the $i$th coordinate of $v$ is finite if and only if $u_i$ is in $\sigma^\perp$.  Then a subset $S$ of $\Trop(Y)$ is a neighborhood of $v$ if and only if it contains every point whose $i$th coordinate is sufficiently close to the $i$th coordinate of $v$, for $u_i$ in $\sigma^\perp$, and whose other coordinates are sufficiently large.  This is the case if and only if $S$ contains some truncated cylinder $C(U,n)$, where $U$ contains $v$.  
\end{remark}

\begin{remark}
One can also describe the topological space $\Trop(Y)$ globally, as a quotient of an open subset of $\R^m$, by tropicalizing Cox's construction of toric varieties as quotients of open subsets of affine spaces \cite{Cox95}, as follows.  First, consider the case where the rays of $\Delta$ span $N_\RR$.  Let $\Delta(1)$ be the set of rays of $\Delta$, and let $\sigma(1)$ be the subset consisting of rays of $\sigma$, for each cone $\sigma \in \Delta$.  Let $\Delta'$ be the fan in $\RR^{\Delta(1)}$ whose maximal cones are of the form $\RR_{\geq 0}^{\sigma(1)}$, for maximal cones $\sigma$ in $\Delta$, with $Y'$ the corresponding invariant open subvariety of $\A^{\Delta(1)}$.  The natural projection $\RR^{\Delta(1)} \rightarrow N_\RR$ taking a standard basis vector to the primitive generator of the corresponding ray induces a map of toric varieties $\varphi: Y' \rightarrow Y$.  We claim that $\Trop(\varphi)$ is surjective and $\Trop(Y)$ carries the quotient topology.  Say $v$ is a point in $N(\sigma) \subset \Trop(Y)$.  Then $\RR^{\Delta(1) \smallsetminus \sigma(1)}$ surjects onto $N(\sigma)$, since the rays of $\Delta$ span $N_\RR$.  It remains to show that $\Trop(Y)$ carries the quotient topology.

Let $S$ be a subset of $\Trop(Y)$ containing $v$ such that $\varphi^{-1}(S)$ is open, and let $v'$ be a preimage of $v$ in $\RR^{\Delta(1) \smallsetminus \sigma(1)}$.  To show that $\Trop(Y)$ carries the quotient topology, we must show that $S$ is a neighborhood of $v$.  Now $\varphi^{-1}(S)$ contains a basic open neighborhood $C(U', n')$ of $v'$.  Linear projection maps $U'$ onto an open subset $U \subset N(\sigma)$, and $S$ contains $C(U,n)$, for $n$ sufficiently large.  Therefore $S$ is a neighborhood of $v$, as required.

If the rays of $\Delta$ do not span $N_\R$, then let $\Delta_0$ be the fan given by $\Delta$ in the span of $\Delta(1)$, with $Y_0$ the corresponding toric variety.  Any choice of splitting $N_\RR \cong \span (\Delta(1)) \times \RR^k$ induces identifications $Y \cong Y_0 \times \G_m^k$ and $\Trop(Y) \cong \Trop(Y_0) \times \RR^k$, making $\Trop(Y)$ a quotient of an open subset of $\R^{\Delta(1)} \times \RR^k$.
\end{remark}

Tropicalization of toric varieties is functorial with respect to arbitrary equivariant morphisms, such as inclusions of invariant subvarieties.  To see this functoriality, it is convenient to work with extended monoids $\S_\sigma = S_\sigma \cup \infty$, where $u + \infty = \infty$ for all $u$, and pointed monoid homomorphisms that take $\infty$ to $\infty$.  Any monoid homomorphism from $S_\sigma$ to $\R$ extends uniquely to a pointed morphism on $\S_\sigma$, so there is a natural identification $\Hom(\S_\sigma, \R) = \Trop(U_\sigma)$.  Suppose $U_\tau$ is an affine toric variety with dense torus $T'$, and $M'$ is the lattice of characters of $T'$.  If $\varphi: U_\tau \rightarrow U_\sigma$ is an equivariant morphism then pulling back regular functions gives a monoid map $\varphi^*: \S_\sigma \rightarrow \S_\tau$, where $\varphi^* (u)$ is defined to be $\infty$ if the pullback of the corresponding regular function vanishes on $U_\tau$.  This map of monoids induces a continuous map of $\Hom$ spaces
\[
\Trop(\varphi): \Trop(U_\tau) \rightarrow \Trop(U_\sigma),
\]
taking a monoid homomorphism $\phi_v: \S_\tau \rightarrow \R$ to $\phi_v \circ \varphi^*$.  Now, if $\varphi': Y' \rightarrow Y$ is an equivariant map of toric varieties, then each invariant affine open subvariety of $Y'$ maps into some invariant affine open subvariety of $Y$, and the induced tropicalization maps for the restrictions of $\varphi'$ to invariant affine opens glue together to give a canonical map from $\Trop(Y')$ to $\Trop(Y)$.   

We now generalize this tropicalization construction to closed subvarieties of toric varieties.

\begin{definition}
Let $X$ be a variety over $K$, and let $\iota: X \hookrightarrow Y(\Delta)$ be a closed embedding.  Then the tropicalization $\Trop(X, \iota)$ is the closure of the image of $X(K)$ in $\Trop(Y)$.
\end{definition}

\noindent When the embedding is fixed, write simply $\Trop(X)$ for the tropicalization of $X \subset Y$.

Basic results about tropicalizations of subvarieties of tori extend in a straightforward way to these extended tropicalizations of subvarieties of toric varieties.  For instance, if $\jmath: X \rightarrow Y(\Delta)$ is a toric embedding, then $X^\an$ is covered by the analytifications of the embedded affine spaces $X \cap U_\sigma$, and the corresponding projections glue to give a proper continuous map $\pi_\jmath : X^\an \rightarrow \Trop(X, \jmath)$.  We now generalize the basic results linking tropicalization to initial forms and degenerations to these extended tropicalizations.

Let $R$ be the valuation ring in $K$, with maximal ideal $\m$ and residue field $k = R/\m$.  Recall that to each point $v \in N(\sigma)$ we associate the tilted group ring $R[M]^v$, whose elements are Laurent polynomials $f = a_1 x^{u_1} + \cdots + a_r x^{u_r}$ such that $\nu(a_i) \geq \<u_i, v\>$.  The initial form $\init_v(f)$ is the image of $f$ in $k[M]^v = R[M]^v \otimes_R k$.  If $X \subset T$ is a closed subvariety, then the tropical degeneration $X_v$ is the $k$-subvariety of $T$ cut out by the initial forms of all Laurent polynomials in $I(X) \cap R[M]^v$.  See \cite{tropicalfibers} for further details.

\begin{proposition} \label{equivalent conditions}
Let $\sigma$ be a cone in $\Delta$, and let $v$ be a $G$-rational point in $N(\sigma)$.  Then the following are equivalent:
\begin{enumerate}
\item The extended tropicalization $\Trop(X)$ contains $v$.
\item The ordinary tropicalization $\trop(X \cap T_\sigma)$ contains $v$.
\item There is a point $x \in X(K)$ such that $\Trop(x) = v$.
\item The tropical degeneration $(X \cap T_\sigma)_v$ is nonempty.
\item For every $f$ in $I(X \cap T_\sigma) \cap R[M]^v$, the initial form $\init_v(f)$ is not a monomial.
\end{enumerate}
\end{proposition}

\begin{proof}
The equivalence of (2)--(5) is standard; see \cite{SpeyerSturmfels04} and \cite{tropicalfibers}.   And (3) implies (1) by definition.  We now show that (1) implies (2).

Suppose $v$ is in the extended tropicalization $\Trop(X)$.  The projection from $X^\an$ to $\Trop(X)$ is proper with dense image, and hence surjective, so we can choose a point $x$ in $X^\an$ whose image in $\Trop(X)$ is equal to $v$.  Then the multiplicative seminorm $|f|_x$ vanishes for any function $f \in K[X \cap U_\sigma]$ that vanishes on $X \cap T_\sigma$.  So $x$ is in the analytification $(X \cap T_\sigma)^\an \subset X^\an$.  Therefore, $v$ is in the image of $(X \cap T_\sigma)^\an$ and hence must be in $\trop(X \cap T_\sigma)$, as required.
\end{proof}

\begin{corollary}
If $V$ is a $T$-invariant subvariety of $Y$, then
\[
\Trop(X) \cap \Trop(V) = \Trop(X \cap V).
\]
\end{corollary}

\section{Analytification of quasiprojective varieties} \label{quasiproj section}

Let $X$ be a quasiprojective variety over $K$.

\begin{definition}
A quasiprojective toric embedding $\iota: X \hookrightarrow Y$ is a closed embedding of $X$ in a quasiprojective toric variety.
\end{definition}

\noindent A morphism of quasiprojective toric embeddings from $\iota$ to $\jmath: X \hookrightarrow Y'$ is an equivariant map $\varphi: Y' \rightarrow Y$ such that $\varphi \circ \jmath = \iota$.  Such a morphism induces a natural map of tropicalizations
\[
\Trop(\varphi) : \Trop(X, \jmath) \rightarrow \Trop(X, \iota),
\]
making $\Trop$ a functor from toric embeddings to topological spaces.  Recall that there are natural proper and continuous maps $\pi_\iota: X^\an \rightarrow \Trop(X, \iota)$, compatible with the tropicalizations of equivariant morphisms.

\begin{theorem} \label{quasiprojective analytification}
Let $X$ be a quasiprojective variety over $K$.  Then $\varprojlim \pi_\iota$ maps $X^\an$ homeomorphically onto $\varprojlim \Trop(X, \iota)$, where the limit is taken over all quasiprojective toric embeddings of $X$.
\end{theorem}

\noindent The proof of Theorem~\ref{quasiprojective analytification} is similar to the proof of Theorem~\ref{main}, given the following lemma which says, roughly speaking, that a quasiprojective variety has enough quasiprojective toric embeddings.

\begin{lemma} \label{qpembeddings}
Let $\overline X$ be a projective variety, with $V \subset \overline X$ a closed subscheme and $U \subset \overline X$ the complement of an effective ample divisor that contains $V$.  Then, for any generators $f_1, \ldots, f_r$ for $K[U]$, there is a closed embedding $\iota: \overline X \hookrightarrow \P^m$ such that
\begin{enumerate}
\item The open subvariety $U$ is the preimage of $\A^m$.
\item The function $f_i$ is the pullback of $x_i \in K[\A^m]$.
\item The closed subvariety $V$ is the preimage of a coordinate linear subspace.
\end{enumerate}
\end{lemma}

\begin{proof}
Let $D$ be an effective ample divisor whose support is exactly $\overline X \smallsetminus U$.  Choose a sufficiently large integer $n$ such that $nD$ is very ample, the rational functions $f_1, \ldots, f_r$, extend to regular sections of $\O(nD)$, and $\I_V \otimes \O(nD)$ is globally generated, where $\I_V$ is the ideal sheaf of $V$.  Let $\iota: \overline X \hookrightarrow \P^m$ be the embedding corresponding to a generating set $\{ s_0, \ldots, s_m \}$ for the space of global sections of $\O(kD)$, where $s_0 = 1$, $s_i = f_i$ for $1 \leq i \leq r$, and some subset of the remaining sections generate $\I_V \otimes \O(nD)$.  Then $U$ is the preimage of the invariant affine open $\A^m$ where $s_0$ does not vanish, $f_i$ is the pullback of $x_i$, and $V$ is the preimage of the coordinate subspace cut out by the $s_i$ that generate $\I_V \otimes \O(nD)$, and the lemma follows.
\end{proof}

\noindent Note that if $\iota$ is an embedding of $\overline X$ in which $V$ is the preimage of a coordinate linear subspace, then the complements of the coordinate hyperplanes containing $V$ are an affine open cover of the quasiprojective variety $\overline X \smallsetminus V$.  We use this cover in the following proof of Theorem~\ref{quasiprojective analytification}.

\begin{proof}[Proof of Theorem~\ref{quasiprojective analytification}]
Let $X$ be a quasiprojective variety.  Choose a projective compactification $\overline X \supset X$, and let $V = \overline X \smallsetminus X$.  Let $U \subset \overline X$ be the complement of an effective ample divisor on $\overline X$ that contains $V$.

By Lemma~\ref{qpembeddings}, there is a closed embedding $\iota:\overline X \rightarrow \P^m$ such that $U$ is the preimage of $\A^m$ and $V$ is the preimage of a coordinate subspace.  Then $\iota$ restricts to a closed embedding of $X$ in an invariant open $Y \subset \P^m$ that contains $\A^m$.  We claim that $\pi_\iota$ maps $U^\an$ bijectively onto the preimage $\mathcal U$ of $\Trop(U, \iota)$ in $\varprojlim \Trop(X, \jmath)$.  The theorem follows from this claim, as we now explain.  First, the topology on the analytification is the coarsest such that $\varprojlim \pi_\jmath$ is continuous, so it will suffice to show that $X^\an$ maps bijectively onto $\varprojlim \Trop(X, \jmath)$.  To see surjectivity, we can replace $U$ by the complement $U_0$ of any coordinate hyperplane containing $V$, and the tropicalizations of these affine open subsets cover $\Trop(X, \iota)$.  Hence if $U^\an$ surjects onto $\mathcal U$ then $X^\an$ surjects onto $\varprojlim \Trop(X, \jmath)$.  Similarly, to see injectivity, if any two points in $X^\an$ have the same image in $\varprojlim \Trop(X, \jmath)$ then they have the same image in the tropicalization of one of these affine open subsets $U_0$.  Then both of these points are in $U_0^\an$ and hence they must be equal.  It remains to show that $U^\an$ maps bijectively onto $\mathcal U$.

First, we show that $U^\an$ injects into $\mathcal U$.  Let $x$ and $x'$ be points in $U^\an$ with the same image in $\mathcal U$.  By Lemma~\ref{qpembeddings}, for any function $f$ in the coordinate ring $K[U]$ we can choose a toric embedding $\iota$ of $X$ such that $U$ is the preimage of $\A^m$ and $f$ is the pullback of a coordinate linear function.  Then $|f|_x$ depends only on $\pi_\iota(x)$.  In particular, if $\pi_\iota(x) = \pi_\iota(x')$ for every $\iota$, then $|f|_x = |f|_{x'}$ for every $f \in K[U]$, and hence $x = x'$.  So $U^\an$ injects into $\mathcal U$, as claimed.

Finally, we show that $U^\an$ surjects onto $\mathcal U$.  Let $y$ be a point in $\mathcal U$. For any $f \in K[U]$, choose an embedding of $X$ in an invariant open subset of $\P^m$ such that $U$ is the preimage of $\A^m$ and $f$ is the pullback of $x_1$.  There is a point $x$ in $U^\an$ defined by setting $|f|_x$ equal to the exponential of the negative of the first coordinate of $y_\iota \in \R^m$.  For any two such embeddings $\iota: X \hookrightarrow Y$ and $\jmath: X \hookrightarrow Y'$, we can take the product $\iota \times \jmath: X \rightarrow Y \times Y'$, and $y_{\iota \times \jmath}$ projects to both $y_\iota$ and $y_\jmath$ in the inverse system, and it follows that $| \ |_x$ is well-defined.  By construction, $x$ is a point in $U^\an$ that maps to $y$.  Therefore, $U^\an$ surjects onto $\mathcal U$ as claimed, and the theorem follows.
\end{proof}

\section{Fields with trivial valuation}

We now consider tropicalizations and analytifications for varieties over an algebraically closed field $k$ equipped with the trivial valuation $\nu(k^*) \equiv 0$.  The geometry in this case remains interesting; for instance, if $k = \C$ then the singular cohomology of the analytification of a complex variety $X$ with respect to the trivial valuation on $\C$ is naturally isomorphic to the weight zero part of the mixed Hodge structure on $H^*(X(\C), \Q)$ \cite[Theorem~1.1(c)]{Berkovich00}.  The techniques and results for fields with nontrivial nonarchimedean valuations extend in a straightforward to fields with the trivial valuation, as follows.

Let $k$ be an algebraically closed field equipped with the trivial valuation, and let $X$ be a  closed subvariety of the torus $T^m$ over $k$.  Let $v$ be a point in $\RR^m$.  By definition, the tilted group ring $k[T]^v$ consists of Laurent polynomials $f = a_1 x^{u_1} + \cdots + a_r x^{u_r}$ with $a_i \in k^*$ such that $\<u_i, v \>$ is nonnegative for all $i$, and the initial form $\init_v(f)$ is the sum of those $a_i x^{u_i}$ such that $\<u_i, v\>$ is zero, for $f \in k[T]^v$.

Suppose $v$ is rational and nonzero, and $\rho$ is the ray spanned by $v$.  Then $k[T]^v$ is the coordinate ring of the affine toric variety $U_\rho$ corresponding to $\rho$.  If $f$ is in $k[T]^v$, then $\init_v(f)$ is canonically identified with the restriction of $f$ to the invariant divisor $D_\rho$ in $U_\rho$.  In particular, by the Nullstellensatz, if $X$ is a closed subvariety of $T$ then the tropical degeneration $X_v$, which is cut out by the initial forms of all functions $f$ in $I(X) \cap k[T]^v$, is nonempty if and only if the closure of $X$ in $U_\rho$ meets $D_\rho$.

\begin{definition} \label{trivial tropicalization}
The tropicalization $\trop(X)$ is the set of $v$ in $N_\RR$ such that $X_v$ is nonempty.
\end{definition}

\noindent This definition agrees with the definition of tropicalization over fields with nontrivial valuation, by Proposition~\ref{equivalent conditions}, and is the underlying set of a rational fan \cite{BJSST}.  The standard argument shows that $\trop(X)$ is the set of $v$ such that $\init_v(f)$ is not a monomial for every $f$ in the ideal of $X$.  Just as for affine varieties over fields with nontrivial valuations, the analytification $X^\an$ is the set of multiplicative seminorms on the coordinate ring $k[X]$ that are uniformly equal to one on $k^*$, equipped with the coarsest topology such that $x \mapsto |f|_x$ is continuous for every $f \in k[X]$. Evaluation of seminorms on the restrictions of characters induces a proper continuous map $\pi:X^\an \rightarrow N_\RR$.

\begin{proposition} \label{BieriGroves}
The image $\pi(X^\an)$ is exactly $\trop(X)$.
\end{proposition}

\begin{proof}
Let $x$ be a point in the analytification $X^\an$.  The usual proof of the ultrametric inequality for nonarchimedean norms shows that $|f + g|_x$ is equal to the maximum of $|f|_x$ and $|g|_x$ if $|f|_x \neq |g|_x$.  Now, suppose $f = a_1 x^{u_1} + \cdots + a_r x^{u_r}$ is in $I(X) \cap k[T]^v$.  Since $f$ is in the ideal of $X$, $|f|_x$ is zero, but the seminorm of each monomial is positive, so there must be at least two monomials in $f$ of maximal norm. It follows that the initial form $\init_{\pi(x)}(f)$ is not a monomial, and hence $\pi(x)$ is in $\trop(X)$.

It remains to show that $\trop(X)$ is contained in the image of $\pi$.  Since $\pi$ is proper and its image is invariant under multiplication by positive scalars, and since $\trop(X)$ is the underlying set of a rational fan, it will suffice to show that any rational ray in $\trop(X)$ is spanned by a point in the image of $\pi$.  Let $\rho$ be a rational ray in $\trop(X)$.  Then the closure of $X$ in $U_\rho$ intersects $D_\rho$.  Let $\nu$ be a valuation centered in $\overline X \cap D_\rho$.  The order of vanishing of a monomial $x^u$ along $D_\rho$ is $\<u,v_\rho\>$, so $\nu(x^u)$ is positive if and only if $\<u, v_\rho\>$ is positive.  It follows that the image of the multiplicative seminorm $\exp(-\nu) \in X^\an$ spans $\rho$, as required.
\end{proof}

We now consider extended tropicalizations of subvarieties of toric varieties over fields with the trivial valuation.

\begin{definition}
Let $\iota: X \hookrightarrow Y(\Delta)$ be a closed embedding in a toric variety over $k$.  Then the extended tropicalization $\Trop(X, \iota)$ is the disjoint union of the tropicalizations $\trop(X \cap T_\sigma)$, for $\sigma \in \Delta$.
\end{definition}

\noindent  Let $K$ be an algebraically closed extension of $k$ which is complete with respect to an extension of the trivial valuation on $k$.  Then, by Propositions~\ref{equivalent conditions} and \ref{base change}, $\Trop(X, \iota)$ is equal to the extended tropicalization of the base change $\Trop(X_K, \iota_K)$.  If $X$ is not necessarily affine, the analytification $X^\an$ is constructed by gluing the analytifications of its affine open subvarieties in the canonical way, and there is a natural continuous and proper map $\pi_\iota: X^\an \rightarrow \Trop(X, \iota)$.   In the affine case, this projection takes $x \in U_\sigma^\an$ to the monoid homomorphism $[u \mapsto -\log | \chi^u|_x],$ for $u$ in $\sigma^\vee \cap M$.  

\begin{theorem}
Let $X$ be an affine \emph{(}resp. quasiprojective\emph{)} variety over $k$.  Then $\varprojlim \pi_\iota$ maps $X^\an$ homeomorphically onto $\varprojlim \Trop (X, \iota)$, where the limit is taken over all affine embeddings $\iota:X \hookrightarrow \A^m$ \emph{(}resp. quasiprojective toric embeddings $\iota: X \hookrightarrow Y(\Delta)$\emph{)}.
\end{theorem}

\begin{proof}  Similar to the proofs of Theorems~\ref{main} and \ref{quasiprojective analytification}, since the image of $X^\an$ in $\R^m$ (resp. $\Trop(Y)$) is exactly $\Trop(X,\iota)$.  \end{proof}

\section{Appendix: Invariance of tropicalization under field extensions}

Here we show that tropicalization is invariant under extensions of valued fields.  This is straightforward in the case where the base field has a nontrivial valuation, but we have not been able to find a reference in the case where the base field has the trivial valuation.  Here we give a brief unified treatment of the general case.  In this appendix, since we consider only tropicalizations and not analytifications, we do not require the fields $k$ and $K$ to be complete with respect to their valuations.

Let $k$ be an algebraically closed field with a valuation that may or may not be trivial.  Let $K$ be an algebraically closed extension of $k$ with a valuation that extends the given valuation on $k$.

\begin{proposition} \label{base change}
Let $X$ be a subvariety of $T$ over $k$.  Then $\trop(X)$ is equal to $\trop(X_K)$.
\end{proposition}

\noindent It is straightforward to see that $\trop(X_K)$ is contained in $\trop(X)$, as follows.  Suppose $v$ is in $\trop(X_K)$.  Then the initial form $\init_v(f)$ is not a monomial, for every function $f$ in the ideal of $X_K$.  Now the ideal of $X$ is contained in the ideal of $X_K$, so it follows that $v$ is in $\trop(X)$.  If the valuation on $k$ is nontrivial, then the reverse containment is also easy, since the image of $X(k)$ in $N_\R$ is contained in the image of $X(K)$.

We now show the reverse containment in the hypersurface case.

\begin{lemma} \label{hypersurface}
Let $X$ be a hypersurface in $T$ over $k$.  Then $\trop(X)$ is contained in $\trop(X_K)$.
\end{lemma}

\begin{proof}
Since $\trop(X_{K'})$ is contained in $\trop(X_{K})$ for any extension of valued fields $K'$ over $K$, we may assume that the valuation on $K$ is nontrivial.  Let $f = a_1 x^{u_1} + \cdots + a_r x^{u_r}$ be a defining equation for $X$ with coefficients in $k$.  Then $\trop(X)$ is contained in the corner locus of the piecewise linear function $\Psi_f$ on $N_\R$ defined by
\[
\Psi_f(w) = \min \{ \<u_1, w\> + \nu(a_1) , \ldots, \<u_r, w\> + \nu(a_r) \};
\]
if $v$ is not in this corner locus and $\Psi_f(v)$ is equal to $\<u_i,v\> + \nu(a_i)$, then the initial form of another defining equation $\init_v\left( \frac{f}{a_ix^{u_i}}\right)$ is equal to one, so $v$ is not in $\trop(X)$.  Standard arguments show that the image of $X(K)$ is dense in the corner locus of $\Psi_f$ \cite[Theorem~2.1.1]{EKL}, so this corner locus is contained in $\trop(X_K)$, and the lemma follows.
\end{proof}

\noindent  We now prove Proposition~\ref{base change} by reducing to the hypersurface case, using a general projection of tori in the sense of \cite[Section~5]{tropicalfibers}; a similar method of reduction was used by Bieri and Groves  \cite{BieriGroves84}.

\begin{proof}[Proof of Proposition~\ref{base change}]
Both $\trop(X)$ and $\trop(X_K)$ are underlying sets of finite polyhedral complexes of pure dimension $\dim X$.  After choosing a polyhedral complex on each and subdividing, we may assume that $\trop(X_K)$ is a subcomplex of $\trop(X)$.  Then we can choose a general rational projection $\phi: N_\R \rightarrow N'_\R$ to a vector space of dimension $\dim X + 1$ corresponding to a surjection of lattices $N \rightarrow N'$ such that the image of each maximal cell in $\trop(X)$  has codimension one in $N'_\R$ and the images of any two distinct cells intersect in codimension at least two.  Since $\trop(X)$ contains $\trop(X_K)$, as noted above, it follows that $\trop(X)$ is equal to $\trop(X_K)$ if and only if $\phi(\trop(X))$ is contained in $\phi(\trop(X_K))$.

The map of vector spaces $\phi$ corresponds to a split surjection of tori $\varphi: T \rightarrow T'$, and $\phi(\trop(X_K))$ is equal to $\trop(X')$, where $X'$ is the closure of $\varphi(X)$ \cite[Corollary~4.5]{tropicalfibers}.   Now $\phi(\trop(X))$ is contained in $\trop(X')$, since $\init_v(\varphi^*f)$ is equal to $\init_{\phi(v)}(f)$, for $f \in k[T']$ and $v \in N_\R$.  And $X'$ is a hypersurface, so $\trop(X')$ is equal to $\trop(X'_K)$, by Lemma~\ref{hypersurface}.  It follows that $\phi(\trop(X))$ is contained in $\phi(\trop(X_K))$, as required.
\end{proof}

\bibliography{math}
\bibliographystyle{amsalpha}

\end{document}